\documentclass[a4paper,11pt]{amsart}

\usepackage{amsmath}
\usepackage{verbatim}
\usepackage{url}

\usepackage[all, cmtip,arrow]{xy}

\usepackage{pb-diagram,pb-xy}
\usepackage{amsthm}

\newcommand{\hs}{\kern 0.75pt}
\newcommand{\hm}{\kern -0.75pt}

\usepackage{amssymb}

\numberwithin{equation}{section}

\newtheorem{thm}[equation]{Theorem}
\newtheorem{prop}[equation]{Proposition}

\newtheorem{lem}[equation]{Lemma}
\newtheorem{cor}[equation]{Corollary}

\newtheorem{mainthm}{Main Theorem}

\theoremstyle{definition}

\newtheorem{exple}[equation]{Example}

\newtheorem{rem}[equation]{Remark}

\newcommand{\sgn}{\operatorname{sgn}}

\newcommand{\Z}{\mathbb{Z}}
\newcommand{\F}{\mathbb{F}}

\newcommand{\Q}{\mathbb{Q}}

\newcommand{\HH}{\mathbb{H}}

\newcommand{\R}{\mathbb{R}}

\newcommand{\Spec}{\operatorname{Spec}}

\newcommand{\disc}{\operatorname{disc}}

\newcommand{\upv}{{^v\!}}
\newcommand{\Fbar}{{\overline{F}}}
\newcommand{\Fbarv}{{\overline{F}\!_v}}

\newcommand{\an}{{\mathrm{an}}}
\newcommand{\ve}{{\varepsilon}}

\newcommand{\into}{\hookrightarrow}

\begin{document}

\title{Hasse principle for Rost motives}

\author{Mikhail Borovoi}

\address{Borovoi: Raymond and Beverly Sackler School of Mathematical Sciences,
       Tel Aviv University, 6997801 Tel Aviv, Israel}

\email{borovoi@post.tau.ac.il}

\thanks{Borovoi was partially supported by the Hermann Minkowski Center for Geometry and by the Israel Science Foundation (grant No.~870/16). Semenov was partially supported by the SPP 1786 ``Homotopy theory and algebraic geometry'' (DFG)}

\author{Nikita Semenov}

\address{Semenov: Mathematisches Institut, Ludwig-Maximilians-Universit\"at M\"unchen,
Theresienstr. 39, D-80333  M\"unchen, Germany}

\email{semenov@math.lmu.de}

\author{Maksim Zhykhovich}

\address{Zhykhovich: Mathematisches Institut, Ludwig-Maximilians-Universit\"at M\"unchen,
Theresienstr. 39, D-80333  M\"unchen, Germany}

\email{zhykhovi@math.lmu.de}

\begin{abstract}
We  prove a Hasse principle for binary direct summands of the Chow motive of a smooth projective quadric $Q$ over a number field $F$.
Besides, we show that such summands are twists of Rost motives.
In the case when $F$ has at most one real embedding we describe a complete motivic decomposition of $Q$.
\end{abstract}


\maketitle

\section{Introduction}

Chow motives were introduced by Grothendieck, and since then they
became a fundamental tool for investigating the structure of algebraic varieties.
Applications of the Chow motives include, among others, results
on higher Witt indices of quadratic forms \cite{Ka03}, structure of the powers of the fundamental ideal
in the Witt ring \cite{Ka04}, cohomological invariants of algebraic groups (\cite{GPS16} and  \cite{S16}), Kaplansky's problem on the $u$-invariants of fields \cite{Vi07}, and isotropy of involutions \cite{KaZ13}.

With a quadratic form $q$ over a field one can associate two discrete invariants, namely, its splitting pattern and its motivic pattern.
The {\em splitting pattern} is the family of all Witt indices of $q$ over all field extensions of the base field,
and the {\em motivic pattern} (or motivic decomposition type, see \cite{Vi11}) describes
the complete motivic decomposition of the respective projective quadric $q=0$.
As far as we know, to determine their possible values is an open problem even if the base field is a number field.
Indeed, in order to determine them, one needs to consider arbitrary field extensions of the base field,
which can be infinite, and thus, are not number fields anymore.

Another issue of motivic nature which is, as far as we know, open even for quadrics over number fields, is the Zariski birationality problem described in \cite{Ohm96}.

In the present article we address the problem of the structure of the Chow motives of smooth projective quadrics over number fields.
Note that Hoffmann  \cite{Ho15} considered a related problem, namely, motivic isomorphisms between quadrics over number fields
(and over real fields satisfying effective diagonalization).

The classical Hasse--Minkowski theorem \cite[Theorem~66:1]{Meara} says that a non-degenerate quadratic form $q$ over a number field $F$ is isotropic
if and only if it is isotropic over all completions of $F$.
This assertion can be reformulated in the language of Chow motives.
Namely, for the projective quadric $Q$ given by the equation $q=0$, the Tate motive $\Z(0)$ is a direct summand of the motive of $Q$
if and only if it is a direct summand of the motive of $Q_{F_v}$ for each completion $F_v$ of $F$.

Moreover, the Hasse--Minkowski theorem readily implies that for every non-negative integer $m$,
the Tate motive $\Z(m)$ is a direct summand of the motive of $Q$ over $F$
if and only if it is a direct summand of the motive of $Q$ over all completions of $F$.
Indeed, this is equivalent to the condition that the Witt index of $q$ is greater  than $m$.
(Recall that the \emph{Witt index} of a quadratic form $q$ on a vector space is the dimension of a maximal totally isotropic subspace.)

In the present article we prove a generalization
of the above assertion replacing the Tate motive $\Z(m)$ by a {\it binary motive}.

We say that a motive $N$ over an arbitrary field $F$ is a {\em split} motive (resp. a {\em binary split} motive) if it is
a direct sum of a finite number of Tate motives over $F$ (resp. if it is a direct sum of {\em two} Tate motives over $F$).
We say that $N$ is a {\em binary} motive over $F$ if becomes binary split over an algebraic closure $\Fbar$ of $F$.

We work in the category of Chow motives with $\F_2$-coefficients (see \cite{Ma68}, \cite{EKM}).
We consider only non-degenerate quadratic forms (see \cite[7.A]{EKM}).
For a non-degenerate quadratic form $q$ over a field $F$
we denote by $M(q)$ the Chow motive of the corresponding projective quadric given by the equation  $q=0$.
For a number field $F$ and a place $v$ of $F$, we denote by $q_v$ the corresponding quadratic form over the completion $F_v$ of $F$ at $v$.

Our main result is the following theorem:

\begin{mainthm}[Theorem~\ref{maintheorem}]
\label{maintheorem-0}
Let $q$ be a quadratic form over a number field $F$.
Let $N$ be a binary split motive over $F$.
Assume that for every place $v$ of $F$ there exists a direct summand \ $\upv M$ of $M(q_v)$ that is isomorphic to $N$ over an algebraic closure $\Fbar_v$ of $F_v$.
Then there exists a direct  summand $M$ of $M(q)$ that is isomorphic to $N$ over $\Fbar$.
\end{mainthm}

It follows from our proof of Main Theorem~\ref{maintheorem-0} that every indecomposable binary direct summand of a quadric over a number field is a twist of a Rost motive; see Corollary~\ref{c1}.
Recall that the motive of a Pfister form $\pi$ over an arbitrary field is isomorphic to a direct sum of twists of one binary motive, which is called the Rost motive of $\pi$ (see \cite{Ro98}).
Rost motives over an arbitrary field appear in the proof of the Milnor conjecture by Voevodsky \cite{Vo03a}, \cite{Vo03b}.

Note that over every completion of a number field, all quadratic forms are excellent,
 and therefore, the structure of their Chow motives is known (see Section~\ref{excellent-dec}).
Moreover, by the results of Vishik \cite{vish-lens} and Haution \cite{Hau},
every motivic decomposition of a quadric with $\F_2$-coefficients can be uniquely lifted to a motivic decomposition with integer coefficients.
Therefore, our main theorem holds for motives with integer coefficients as well.

Finally, as an application of the Hasse principle for binary motives we obtain a complete motivic decomposition of the motive
of a quadric over any number field with at most one real embedding, for example, over the field of rational numbers $\Q$ (see Corollary~\ref{c2}).

\section{Motivic decomposition of excellent forms}
\label{excellent-dec}

In this section $F$ is an arbitrary field of characteristic not $2$.

A quadratic form $q$ over $F$ is called excellent, if for every field extension $E/F$ the anisotropic part of the form $q_E$ is defined over $F$.

With every anisotropic excellent quadratic form $q$ over $F$ one can associate a strictly decreasing sequence of embedded Pfister forms
$\pi_0 \supset \pi_1 \supset  \ldots \supset \pi_r$, where $r$ is a positive integer and
${\dim (\pi_{r-1})> 2\dim(\pi_r)}$ (see \cite[\S7]{Kn77}, \cite[Definition~5.4.9]{Kahn}, \cite[Section~7]{Kar-Mer}).
Recall that an $n$-fold Pfister form is a form of the type $\langle 1, a_1\rangle\otimes\ldots\otimes\langle 1, a_n\rangle$ for some $a_i\in F^*$,
and the excellent form $q$ is similar to the anisotropic part of the form $\pi_0-\pi_1+\ldots+(-1)^r\pi_r$.
Let $\dim(\pi_i) = 2^{n_i}$ for every  $i \in [0,r]$,  $n_{r-1}> n_r +1$. We have
$$ \dim q =2^{n_0}-2^{n_1} + \ldots + (-1)^r2^{n_r}\, .$$

For an anisotropic Pfister form $\pi$ we denote by $R(\pi)$ the Rost motive of $\pi$ (see \cite[Section~4]{Ro98}).
By \cite[Corollary~7.2]{Kar-Mer} we have

\begin{equation}
\label{dec}
M(q)\simeq (\bigoplus_{i=0}^{m_0-1} R(\pi_0)(i))\bigoplus (\bigoplus_{i=m_0}^{m_0+m_1-1} R(\pi_1)(i))\bigoplus
\ldots \bigoplus (\bigoplus_{i=m_0+\ldots+m_{\tilde{r}-1}}^{m_0+\ldots+m_{\tilde{r}}-1} R(\pi_{\tilde{r}})(i)) \, ,
\end{equation}

\noindent where $\tilde{r} = r$ if $\dim q$ is even and $\tilde{r} = r-1$ if $\dim q$ is odd, and
$$m_j=2^{n_j-1} - 2^{n_{j+1}}+ \ldots + (-1)^{j+r}2^{n_r} \,  $$
for every $j \in [0, \tilde{r}]$.
If we want to specify the excellent form $q$, we write $m_j(q)$, $n_j(q)$, $r(q)$ etc.

 More generally, for every $k \in [0,\tilde{r}]$ we have
\begin{equation}
\label{sum0k}
m_0+\ldots+m_{k} = \left\{ \begin{array}{cl}
(\dim\, [q]_k)/2 & \mbox{ if } k \text{ is odd}, \\
(\dim\, [q]_k)/2 - 2^{n_{k+1}}+ \ldots + (-1)^{r}2^{n_r} & \mbox{ if } k \text{ is even},
\end{array} \right.
\end{equation}
where $[q]_k$ is the  excellent quadratic form given by a strictly decreasing sequence of embedded Pfister forms
$\pi_0 \supset \pi_1 \supset \ldots \supset \pi_k$.  It follows that
\begin{equation}
\label{explicit}
m_0+\ldots+m_{k} = (\dim q - | \dim q - \dim \, [q]_k |)/2 \, .
\end{equation}

Note that $m_0 + \ldots + m_{\tilde{r}}= \Big [\frac{\dim q}{2} \Big ]$, and this is the number of indecomposable direct summands in the complete motivic decomposition of $q$.

Moreover, if $\dim q=2d+2$ is even, then $\disc q=\disc\pi_r$, where $\disc$ denotes the discriminant. In particular, $\disc q$ is non-trivial if and only if $\pi_r$ is a binary form with non-trivial discriminant, and by decomposition~\eqref{dec} this happens if and only if the motive of $q$ contains an indecomposable direct summand $N$ which over a splitting field of $q$ becomes isomorphic to $\F_2(d)\oplus\F_2(d)$, and in this case $N\simeq M(\pi_r)(d)\simeq M(\Spec F(\sqrt{\disc q}))(d)$, where for a smooth projective variety $X$ we denote by $M(X)$ its Chow motive.

Finally, for an arbitrary $n$-dimensional quadratic form $q$ over $F$ with the Witt index $m$
we have by \cite[Proposition~2]{Ro98} (see also \cite[Example~66.7]{EKM}) the following motivic decomposition:

\begin{equation}\label{isotr}
M(q)\simeq \F_2\oplus\ldots\oplus \F_2(m-1)\oplus M(q_\an)(m)\oplus\F_2(n-m-1)\oplus\ldots\oplus \F_2(n-2),
\end{equation}
where $q_\an$ is the anisotropic part of $q$.

\section{Construction of some Pfister forms}

In this section $F$ is a number field. We denote by $\Omega(F)$ the set of all places of $F$.
For $v\in\Omega(F)$, we denote by $F_v$ the completion of $F$ at $v$.
For  two non-zero elements $a, b \in F_v$, we denote by $(a,b)_v\in\{\pm1\}$ their Hilbert symbol (see \cite[\S 63.B]{Meara}).
Recall that $(a,b)_v=1$ if $z^2-ax^2-by^2 = 0$ has a non-zero solution $(x,y,z)$ in $F_v^3$, and $(a,b)_v=-1$ otherwise.
If $a,b\in F^*$, we write $(a,b)_v$ for $(a_v,b_v)_v$, where $a_v,b_v\in F_v$ are the images of $a$ and $b$, respectively.

For a diagonal quadratic form $p = \langle a_1, \ldots, a_n \rangle$, $a_i \in F_v^*$, over $F_v$  we define the Hasse invariant $\ve_v(p)$ as in \cite{Lam} and \cite{Ser}:
\begin{equation}
\label{Hasse-invariant}
\varepsilon_v(p) = \prod_{ i < j } (a_i,a_j)_v\in\{\pm1\}.
\end{equation}

Note that if $v$ is a real place, then every quadratic form $f$ over $F_v=\R$ is equivalent to ${x_1^2+\ldots+x_n^2-y_1^2-\ldots-y_m^2}$ for some $n$ and $m$, and the Hasse invariant of $f$ is given by
\begin{equation}
\label{Hasse-invariant-real}
\ve_v(f)= (-1,-1)_v^{\frac{m(m-1)}{2}}= (-1)^{\frac{m(m-1)}{2}} \, .
\end{equation}

Let $q$ be a quadratic form over $F$.
For $v\in\Omega(F)$, the embedding $F \into F_v$ induced by $v$  allows one to regard $q$ as a quadratic form  over $F_v$, which
we denote by $q_v$.

The classical Hasse--Minkowski theorem \cite[Theorem~66:1]{Meara} immediately implies that
\begin{equation}\label{hm}
\dim q_\an=\max_{v\in\Omega(F)} \dim (q_v)_\an \, .
\end{equation}

In this section we construct Pfister forms with  prescribed local behavior. We shall use these forms in the proof of the main theorem.

We start by recalling several well-known facts about quadratic forms over local fields $F_v$,
where $v \in \Omega(F)$ is a {\em finite} place
(see \cite{Lam}, \cite{Meara}).
First of all, a quadratic form $p$ over $F_v$ is determined by its invariants: dimension, determinant $d(p)\in F_v^*/ {F_v^*}^2 $ and the Hasse invariant
$\varepsilon_v(p) \in \{\pm 1 \}$.

Every quadratic form over $F_v$ of dimension greater than $4$ is isotropic (see \cite[63:19]{Meara}), and there exists a unique (up to isomorphism) anisotropic form of dimension $4$.
This form is a Pfister form and therefore has trivial determinant (see \cite[63:17]{Meara}). Its Hasse invariant can be also computed:

\begin{lem}
\label{4dim}
The $4$-dimensional quadratic form $f$ with trivial determinant over $F_v$ is an\-isotropic if and only if its Hasse invariant is equal to $-(-1,-1)_v$.
\end{lem}
\begin{proof}
The split $4$-dimensional form $\langle 1,-1,1,-1 \rangle$ has trivial determinant and the Hasse invariant $(-1,-1)_v$.
Since $f$ has trivial determinant too and $f\not\simeq \langle 1,-1,1,-1 \rangle$, it must have  a different Hasse invariant.
We conclude that $\ve_v(f)=-(-1,-1)_v$.
\end{proof}

The following proposition provides necessary and sufficient conditions for the existence of a quadratic form over $F$ with prescribed local behavior.
\begin{prop}\label{existence}
Let $n$ be a positive integer, and for every $v \in \Omega(F)$ let $p^v$ be an $n$-dimensional quadratic form over $F_v$. For the existence of an $n$-dimensional quadratic form $p$ over $F$ such that $p_v \simeq p^v $ for all $v \in \Omega(F) $, it is necessary and sufficient that the following conditions hold: \\
$\text{ }(1)$ There exists $c \in F^*$ such that $d(p^v)=c_v\cdot F_v^{*2}$ for all $v \in \Omega(F)$. \\
$\text{ }(2)$ $\varepsilon_v(p^v)= 1$ for almost all $v \in \Omega(F)$. \\
$\text{ }(3)$ $\prod_{v \in \Omega(F)} \varepsilon_v(p^v) = 1$.
\end{prop}
\begin{proof}
The statement is proved in \cite[Theorem 72:1]{Meara}, but with $\ve'_v(p^v)$ instead of $\ve_v(p^v)$,
where for a diagonal quadratic form $f=\langle b_1, \ldots, b_n \rangle$, $b_i \in F_v^*$,  O'Meara considers the invariant
\begin{equation}
\label{Hasse-O'Meara}
\varepsilon'_v(f) = \prod_{ i \le j } (b_i,b_j)_v=\ve_v(f)\cdot (d(f),d(f))_v\in\{\pm1\}.
\end{equation}
Note that  by \cite[Theorem 71:18]{Meara}  we have $(c,c)_v = 1$ for almost all $v \in \Omega(F)$ and
\[\prod_{v \in \Omega(F)} (c,c)_v =1.\]
 Therefore, conditions (1--3) of Proposition \ref{existence}
are satisfied for an element $c\in F^*$ and the numbers  $\ve_v(p^v)\in\{\pm1\}$
if and only if they are satisfied for the same $c$ and  the numbers
$$\ve'_v(p^v) = (c,c)_v \cdot\ve_v(p^v) \in\{\pm1\}.$$
This completes the proof of the proposition.
\end{proof}

Recall that a quadratic form of dimension $n$ is called split if its Witt index has the maximal value $[n/2]$.

\begin{prop}
\label{odd-pi}
Let $F$ be a number field. Let $q$ be an anisotropic quadratic form over $F$ of odd dimension $2d+1$.
Assume that for every place $v \in \Omega(F)$ there is a direct summand $^v\!M$ of $M(q_v)$ such that
${^v\!M_{\Fbarv} \simeq \F_2(d-1) \oplus \F_2(d)}$.
Then there exists a $2$-fold Pfister form $\pi$ over $F$
that splits exactly at those places $v$ of $F$ at which $q$ splits.
\end{prop}
\begin{proof}
Since the dimension of the quadratic form $q$ is odd,
without loss of generality we may assume that the form $q$ has trivial determinant.
Indeed, we can replace $q$ by  the form $d(q)\cdot q$, which has trivial determinant and the same motive as $q$ and splits exactly at those places $v$ of $F$ at which $q$ splits.

We apply Proposition \ref{existence} to the family $\mathcal{F}=\{ {\pi}^v \mid \, v \in \Omega(F)\}$ of quadratic forms, where for every place $v \in \Omega(F)$ the form ${\pi}^v$ is defined as follows:
$$ {\pi}^v = \left\{ \begin{array}{l}
\text{the split form $\langle 1,-1,1,-1 \rangle$ of dimension 4, if } q_v \text{ is split}; \\
\mbox{a (unique) anisotropic form of dimension } 4, \text{ if } q_v \text{ is not split and } v \text{ is finite};\\
\langle 1,1,1,1 \rangle, \text{ if } q_v \text{ is not split and } v \text{ is real.}
\end{array} \right.$$
Since $d(\pi^v)=1$ for all $v \in \Omega(F)$, condition~(1) in Proposition~\ref{existence} is satisfied (with $c=1$).

Assume that conditions~(2) and (3) in Proposition~\ref{existence} are also satisfied for the family $\mathcal{F}$. It follows then that there exists a quadratic form $\pi$ such that for every $v \in \Omega(F)$ we have ${\pi}_v \simeq {\pi}^v$. By the Hasse--Minkowski theorem, the quadratic form  $\pi$ represents $1$ and $d(\pi) =1$, because this holds locally for every $v \in \Omega(F)$. Since $\pi$ represents 1, it is equivalent to a diagonal form $\langle 1,a,b,c\rangle$ for some $a,b,c\in F^*$.
Since $d(\pi)=1$, we may take $c=ab$.
We see that
$\pi$ is a $2$-fold Pfister form  $\langle 1,a\rangle\otimes\langle 1,b\rangle$. Clearly, by the construction of $\mathcal{F}$, the form $\pi$ satisfies the conclusion of the proposition.

Finally, the following lemma shows that conditions~(2) and (3) in Proposition \ref{existence} are indeed satisfied for the family $\mathcal{F}$.
\end{proof}

\begin{lem}
Under the hypotheses of Proposition~\ref{odd-pi}, for every $v \in \Omega(F)$ we have ${\varepsilon_v(q)=\varepsilon_v({\pi}^v)}$.
\end{lem}
\begin{proof}
Since $q$ is anisotropic over $F$, there exists a place $w \in \Omega(F)$ such that $q_w$ is anisotropic over $F_w$.
Note that every anisotropic form over $F_w$ is an excellent form.
Therefore, the motivic decomposition of $q_w$ is given by formula~\eqref{dec}.
The existence of the motivic direct summand of $M(q_w)$ as in the statement of Proposition~\ref{odd-pi}
implies that $\dim q$ is either $3$ or $5$ modulo $8$. \\
\noindent Let us now check case by case that for every $v \in \Omega(F)$ we have $\ve_v(q)=\ve_v({\pi}^v)$.\\

\noindent \textit{1st case: $v$ is such that $q_v$ is split.} In this case ${\pi}^v$ is split, and therefore, $\ve_v({\pi}^v) = (-1,-1)_v$.
Since $q_v$ is split and has trivial determinant, we have $q \simeq d \, \HH \perp \langle (-1)^d \rangle$, where $\HH$ denotes the hyperbolic plane.
Thus, $\ve_v(q)= (-1,-1)_v^{\frac{d(d-1)}{2}} (-1, (-1)^d)_v^d$. Since $\dim q = 2d+1$ is $3$ or $5$ modulo $8$, we obtain in both cases $\ve_v(q)= (-1,-1)_v$. \\

\noindent \textit{2nd case: $v$ is finite and $\dim (q_v)_\an=3$.} By Lemma~\ref{4dim}, we have $\ve_v({\pi}^v) = -(-1,-1)_v$.
Since a split quadratic form of dimension $2d+1$ and with trivial determinant has the Hasse invariant $(-1,-1)_v$, for our non-split form we have $\ve_v(q)= -(-1,-1)_v$. \\

\noindent \textit{3rd case: $v$ is real and $\dim (q_v)_\an>1$.} It is clear that $\ve_v({\pi}^v) = 1$.
Taking into account formula~\eqref{Hasse-invariant-real},
 by the same argument as in the beginning of the proof we obtain that $\dim (q_v)_\an$ is either $3$ or $5$ modulo $8$.
Let $m$ denote the number of minuses in the signature of $q_v$.

If $\dim (q_v)_\an \equiv \dim q_v $ modulo $8$, then $d((q_v)_\an) = d(q_v)=1$.
Since $(q_v)_\an$ is of odd dimension, then $(q_v)_\an$ is positive definite. It follows that $m = (\dim q_v - \dim (q_v)_\an)/2$ and that $m$ is divisible by $4$.
By formula~\eqref{Hasse-invariant-real}, $\ve_v(q)= 1$.

If $\dim (q_v)_\an \not{\!\!\equiv} \dim q_v $ modulo $8$, then one of these dimensions is $3$ modulo $8$ and another one is $5$ modulo $8$.
Therefore $q_v=k\HH\perp (q_v)_\an$,
where $k=(\dim q_v -\dim (q_v)_\an)/2$ is odd. It follows that ${d(q_v)=(-1)^k d((q_v)_\an) =-d((q_v)_\an)}$ and, thus, $d((q_v)_\an)=-d(q_v)=-1$. Therefore $(q_v)_\an$ is negative definite.
It follows that $m = (\dim q_v +\dim (q_v)_\an)/2$ and that $m$ is divisible by $4$. By formula~\eqref{Hasse-invariant-real} $\ve_v(q)= 1$.
\end{proof}

\begin{prop}
\label{even-pi}
Let $F$ be a number field. Let $q$ be an anisotropic quadratic form over $F$ of even dimension $2d+2$.
Assume that for every place $v \in \Omega(F)$ there is a direct summand $^v\!M$ of $M(q_v)$ such that
${^v\!M_{\Fbarv}  \simeq \F_2(d-1) \oplus \F_2(d)}$.
Then there exists a $2$-fold Pfister form $\pi$ over $F$
that splits exactly at those places $v$ of $F$ at which $q$ splits.
\end{prop}
\begin{proof}
The proof is similar to that of  Proposition~\ref{odd-pi}. \\
Observe that for every finite place $v\in \Omega(F)$ the condition on $M(q_v)$ implies that $q_v$ is either split or $\dim (q_v)_\an=4$.
We define the family  $\mathcal{F}=\{ {\pi}^v \mid v \in \Omega(F)\}$ as follows:
$$ {\pi}^v = \left\{ \begin{array}{l}
\text{the split form $\langle 1,-1,1,-1 \rangle$ of dimension 4, if } q_v \text{ is split}; \\
\mbox{a (unique) anisotropic form of dimension } 4, \text{ if } q_v \text{ is not split and } v \text{ is finite};\\
\langle 1,1,1,1 \rangle, \text{ if } q_v \text{ is not split and } v \text{ is real.}
\end{array} \right.$$

The lemma below shows that the hypotheses of Proposition \ref{existence} are satisfied for the family $\mathcal{F}$.
It follows that there exists a quadratic form $\pi$ such that for every $v \in \Omega(F)$ we have ${\pi}_v \simeq {\pi}^v$. Note that $\pi$ represents $1$ and $d(\pi)=1$. Therefore, $\pi$ is a $2$-fold Pfister form and satisfies the conclusion of the proposition.
\end{proof}

\begin{lem}
Under the hypotheses of Proposition~\ref{even-pi} for every $v \in \Omega(F)$ we have ${\ve_v(q)=\ve_v({\pi}^v)}$.
\end{lem}
\begin{proof}
Since $q$ is anisotropic over $F$, there exists a place $w \in \Omega(F)$ such that $q_w$ is anisotropic over $F_w$. Note that every anisotropic form over $F_w$ is
an excellent form. Therefore, the motivic decomposition of $q_w$ is given by formula~\eqref{dec}. The existence of the motivic direct summand of $M(q_w)$ as in the statement of Proposition~\ref{even-pi} implies that $\dim q$ is $4$ modulo $8$. In particular, $\disc(q)=d(q)$, where $\disc(q)$ denotes the discriminant of $q$  (recall that, by definition, ${\disc(q)=(-1)^{\frac{n(n-1)}{2}}d(q)}\in F^*/ {F^*}^2 $, where $n= \dim q$).

Note also that for every $v \in \Omega(F)$ by formula~\eqref{dec} the motive $M(q_v)$ does not have an indecomposable direct summand over $F_v$ which becomes $\F_2(d)\oplus\F_2(d)$ over $\Fbarv$.
Therefore, for every $v \in \Omega(F)$ we have $\disc(q_v)=1$ (see Section~\ref{excellent-dec}). Thus by the Hasse--Minkowski theorem in dimension $2$ we have  $\disc(q)=1$ and $d(q)=1$.

Let us now check case by case that for every $v \in \Omega(F)$ we have $\ve_v(q)=\ve_v({\pi}^v)$.\\

\noindent \textit{1st case: $v$ is such that $q_v$ is split.} In this case ${\pi}^v$ is split. Since $2d+2$ is $4$ modulo $8$, we have
$$ \ve_v(q)= (-1,-1)_v^{\frac{d(d+1)}{2}}=(-1,-1)_v=\ve_v({\pi}^v) \, . $$

\noindent \textit{2nd case: $v$ is finite and $\dim (q_v)_\an=4$.} By Lemma~\ref{4dim}, we have $\ve_v({\pi}^v) = -(-1,-1)_v$. Since a split quadratic form of dimension $2d+2$ has trivial determinant and the Hasse invariant $(-1,-1)_v$, the form $q_v$, which is of the same dimension and also has the trivial determinant, has the Hasse invariant $-(-1,-1)_v$. \\

\noindent \textit{3rd case: $v$ is real and $\dim (q_v)_\an>0$.} It is clear that $\ve_v({\pi}^v) = 1$.
Note that by the same argument as in the beginning of the proof we have $\dim (q_v)_\an$ is $4$ modulo $8$.
Let $m$ denote the number of minuses in the signature of $q_v$.

Depending on whether $(q_v)_\an$ is positive definite or not, we have $m = (\dim q_v \pm \dim (q_v)_\an)/2$. Thus, in any case $m$ is divisible by $4$. Then by formula~\eqref{Hasse-invariant-real} we have $\ve_v(q)= 1$.
\end{proof}

\section{Vishik's results}

In this section $F$ is an arbitrary field of characteristic not $2$.
We reformulate two Vishik's results for our convenience in order to use them in the proof of the main theorem.

\begin{prop} [\cite{vish-lens}, Theorem 5.1]
\label{vishik1}
Let $p = f \otimes \pi$ be a quadratic form over $F$, where $\pi$ is an $n$-fold Pfister form and $f$ is an odd-dimensional form. Then $M(\pi)(s)$ is a direct summand of $M(p)$, where $s = (\dim p - \dim \pi)/2$.
\end{prop}

Let $R$ be a geometrically split motive over $F$ and $r$ an integer. We say that $R$ ``starts at $r$'', if $r$ is the minimal integer $i$ such that the Tate motive $\F_2(i)$ is a direct summand of $R_{\overline{F}}$.

\begin{prop}[\cite{vish-lens}, Theorem 4.17]
\label{vishik2}
Let $p$ and $q$ be two anisotropic quadratic forms over a field $F$.
Let $m_p$, $m_q$ be fixed non-negative integers.
Assume that $M(p)$ has an indecomposable direct summand of the form $R(m_p)$, where $R$ is a motive, which ``starts at $0$''.
If for all  field extensions $E/F$ we have
$$
i_W(q_E)>m_q \Longleftrightarrow i_W(p_E)>m_p\, ,
$$
where $i_W(q_E)$ and $i_W(p_E)$ are the corresponding Witt indices over $E$,
then $R(m_q)$ is a direct summand of $M(q)$.
\end{prop}

\section{Proof of the main theorem and its applications}
\label{s:proof-main-theorem}

We say that a motive $N$ over a field $F$ is binary if it becomes isomorphic to a direct sum of two Tate motives over $\overline{F}$. We say that $N$ is a binary split motive if it is a direct sum of two Tate motives over $F$.

\begin{thm}\label{maintheorem}
Let $q$ be a quadratic form over a number field $F$. Let $N$ be a binary split motive over $F$. Assume that for every place $v \in \Omega(F)$ there exists a direct motivic summand $^{v \!}M$ of $M(q_v)$ such that $^{v \!}M_{\Fbarv} \simeq  N_{\Fbarv}$. Then there exists a direct motivic summand $M$ of $M(q)$ such that $M_{\overline{F}} \simeq N_{\overline{F}}$.
\end{thm}

\begin{rem}
It follows from \cite[Theorem~3.11]{vish-lens} that
the Krull--Schmidt principle holds for the Chow motives of quadrics. Thus, the motive $M$ is isomorphic to $\upv M$ over $F_v$ for all $v$.
\end{rem}

\begin{proof}
If the quadratic form $q$ is isotropic, then using formula~\eqref{isotr} we can reduce to the anisotropic part of $q$. Therefore, without loss of generality we may assume that the form $q$ is anisotropic.
We denote by $\Omega_{2}(F)$ (resp. $\Omega_{1+1}(F)$ ) the subset of $\Omega(F)$ consisting of the places $v$ such that $^vM$ is indecomposable (resp. $^vM$ is split).

Note that for every $v \in \Omega(F)$ the form $q_v$ as well as its anisotropic part $(q_v)_\an$ is
an excellent form. Throughout the proof of the theorem we denote simply  by $n_j(v)$, $m_j(v)$ and $r(v)$ respectively the numbers $n_j((q_v)_\an)$, $m_j((q_v)_\an)$ and $r((q_v)_\an)$ from Section~\ref{excellent-dec} corresponding to the motivic decomposition of the anisotropic excellent form $(q_v)_\an$.

Since $q$ is anisotropic, the set $\Omega_{2}(F)$ is not empty. Let $v \in \Omega_{2}(F)$. It follows from decomposition~\eqref{dec} and the Krull--Schmidt principle for the motives of quadrics (see \cite{vish-lens}, \cite{ChM06}) that $^{v \!}M \simeq {^v\!R}(t)$, where $t \geq 0$ and ${^v\!R}$ is the Rost motive of some anisotropic $n$-fold Pfister form over $F_v$, $n \geq 1$. Note that $t$ and $n$ do not depend on $v \in \Omega_{2}(F)$. Indeed, this follows from the hypothesis of the theorem, namely we have    $^{v \!}M_{\Fbarv} \simeq ^{w \!\! \!}M_{\overline{F}\!_w}$ for every $v, w \in \Omega_{2}(F)$.

It follows that for every $v\in \Omega_{2}(F)$ the summand $ \pm 2^n$ is presented in the decomposition of $\dim \,(q_v)_\an$ into an alternating sum of powers of $2$,
$\dim \,(q_v)_\an = \sum_{i=0}^{r(v)}(-1)^i 2^{n_i(v)}.$
We denote by $k(v)$ the index satisfying the equality $n_{k(v)}=n$ in this decomposition. We also define
$$Q(v):=\sum_{i=0}^{k(v)} (-1)^i 2^{n_i(v)} , \quad A(v):=\sum_{i=k(v)+1}^{r(v)} (-1)^{i-k(v)-1} 2^{n_i(v)}$$
i.e. $Q(v)= 2^{n_0}-2^{n_1} +  \ldots \pm 2^n $ is a part (up to $2^n$) of the decomposition of $\dim \,(q_v)_\an$ and
${A(v)= | \dim \,(q_v)_\an - Q(v)|}$. Note that in terms of the notation of Section~\ref{excellent-dec} we also have $Q(v) = \dim \, [(q_v)_\an]_{k(v)}$.

From now on we assume that $n>1$, and we shall consider the case $n=1$ at the very end of the proof.

In order to prove the theorem we shall apply Proposition~\ref{vishik1} to the quadratic form $q$ and a suitable form $p$.
The form $p$ will be constructed below as a product $f \otimes \pi$, where $\pi$ is an $n$-fold Pfister form, and will satisfy the following properties:  \\
1. For every place $v \in \Omega(F)$:
\begin{equation}
\label{prop1}
  {\pi}_v \, \text{is split} \quad \Longleftrightarrow \quad ^{v \!}M \, \text{is split}. \,
\end{equation}
\noindent 2.
For every place $v \in \Omega_{2}(F)$ the following two equalities hold:

\begin{equation}
\label{prop2}
\dim (p_v)_\an =  Q(v) \,
\end{equation}
\noindent and
\begin{equation}
\label{prop3}
    \dim (q_v -p_v)_\an = \, | \dim (q_v)_\an - \dim(p_v)_\an  | \, = A(v) \, .
\end{equation}

\begin{lem} {\bf (Construction of $\pi$)}
There exists an $n$-fold Pfister form over $F$, which satisfies Property~\eqref{prop1} for every $v \in \Omega(F)$.
\end{lem}
\begin{proof}
First, assume that $n=2$. Then for every $v \in \Omega(F)$ the motive $^{v \!}M$ is split if and only if the form $q_v$ is split.
If $\dim q =2d+1$, then $N \simeq \F_2(d-1) \oplus \F_2(d)$\ and the result follows from Proposition~\ref{odd-pi}.

Now consider the case $\dim q=2d +2$. Denote by $S$ the motive $\F_2(d-1) \oplus \F_2(d)$. Then either $N \simeq S$ or $N \simeq S(1)$.
In the first case the conclusion of the lemma follows from Proposition \ref{even-pi} and the second case can be reduced to the first one.
Indeed, if $N \simeq S(1)$, then it follows from the motivic decomposition of the excellent form $q_v$ (see Section~\ref{excellent-dec}) that for every $v\in \Omega(F)$ the motive $^v\!M(-1)$ is also a direct motivic summand of $q_v$.

Assume that $n > 2$. By the Weak Approximation Theorem (see Cassels \cite[Section 15]{Ca67} for the Strong Approximation Theorem, which is stronger),
there exists $a \in F$ such that the following property holds for every $v \in \Omega_{\R}(F)$, where $ \Omega_{\R}(F)$ is the set of all real places of $F$:
$$  {a}_v < 0 \quad \Longleftrightarrow \quad ^{v \!}M \, \text{is split}. \, $$
\noindent Consider the $n$-fold Pfister form $\pi = \big \langle 1, a  \big \rangle \otimes \big \langle 1, 1  \big \rangle \otimes \ldots \otimes \big \langle 1, 1  \big \rangle$. We claim that $\pi$ satisfies Property~\eqref{prop1} for every $v \in \Omega(F)$. Indeed, if $v \in \Omega(F)$ is finite or complex, then the form ${\pi}_v$ and the motive $ ^{v \!}M$ are both split.
Thus, in this case Property~\eqref{prop1} holds. Moreover, by the construction of $a$ the form $\pi$ satisfies Property~\eqref{prop1} for every real place $v \in \Omega(F)$.
\end{proof}

\begin{lem} {\bf (Construction of $f$)}
There exists a quadratic form $f$ over $F$ such that for every place $v \in \Omega_{2}(F)$ the form $p:= f \otimes \pi$ satisfies Properties~\eqref{prop2} and \eqref{prop3}.
\end{lem}
\begin{proof}
Let $\Omega_{2, \R}(F) = \{v_1, \ldots , v_l\}$ be the set of all real places $v \in \Omega(F)$ such that $^{v \!}M$ is indecomposable, i.e. $\Omega_{2, \R}(F) = \Omega_{2}(F) \cap \Omega_{\R}(F)$. For every $i=1,\ldots, l$ we construct a quadratic form $f_i$ over $F$ as follows.\\
We fix an integer $i \in [1,l] $ and denote $v_i$ simply by $v$. Since $v \in \Omega_{2}(F) $, the power $\pm 2^n$ appears in the alternating sum
$$\dim \,(q_{v})_\an = \sum_{j=0}^{r(v)} (-1)^j 2^{n_j(v)}\, .$$
\noindent Note that $Q(v)= \sum_{j=0}^{k(v)} (-1)^j 2^{n_j(v)}$, where $n_{k(v)}(v)=n$,
is divisible by $2^n = \dim \, \pi$. We construct $f_i$ in the form $\big \langle 1, a_i  \big \rangle \otimes \big \langle 1, \ldots , 1  \big \rangle$, $a_i \in F$ such that the following equality holds:
\begin{equation}
\label{f_i}
(\dim \, f_i + (-1)^{k(v)}) (\dim \pi) = Q(v) \, .
\end{equation}
\noindent
By the Weak Approximation Theorem we can choose $a_i \in F$ such that for every $j=1, \ldots, l$ the following property holds:
$$(a_i)_{v_j} > 0 \quad \Longleftrightarrow \quad i=j  \, . $$
We define $f:= \alpha (f_1 \perp \ldots \perp f_l \perp \langle b \rangle)_\an $, where
$\alpha, b \in F$ satisfy respectively the following properties for every $v\in \Omega_{2, \R}(F)$\ :
$$\sgn b_{v} = \sgn (-1)^{k(v)} \, .$$
\noindent and
$$\quad \quad \quad \alpha_v > 0 \quad \Longleftrightarrow \quad (q_v)_\an \text{ is positive definite}. $$
Note that the existence of $\alpha$ and $b$ is once again guaranteed by the Weak Approximation Theorem.

We claim that $p=f \otimes \pi$ satisfies Properties~\eqref{prop2} and \eqref{prop3} for every place $v \in \Omega_{2}(F)$.

Indeed, let $v=v_i \in \Omega_{2,\R}(F) $. Note that all forms $f_j$, $j\in \{1, \ldots ,l\} $, $j \neq i$, become hyperbolic over $v$. Thus, we have in the Witt ring $W(F_v)$:
$$p_v = \alpha_{v}( (f_i)_v + \langle b_v \rangle)({\pi}_v) = \pm (( \dim f_i){\langle 1 \rangle} + {\langle (-1)^{k(v)} \rangle})({\pi}_v) \, .$$
It follows from formula~\eqref{f_i} that $\dim \,(p_v)_\an = Q(v)$. By the construction of $\alpha$, equality~\eqref{prop3} also holds.

Assume now that there is a finite place $v \in \Omega_{2}(F)$. It is possible only if $n=2$.
Since the dimension of $f$ is odd and $\pi_v$ is anisotropic, we have $(p_v)_\an = (f_v \otimes \pi_v)_\an \simeq \pi_v$. We also have $\dim (q_v)_\an = 4$, $Q(v)=4$, $A(v)=0$ if $\dim q$ is even, and $\dim (q_v)_\an = 3$, $Q(v)=4$, $A(v)=1$ if $\dim q$ is odd. In both cases the equality $\dim(p_v)_\an= Q(v)$ holds.

Recall that over $F_v$ there is a unique anisotropic quadratic form of dimension $4$. Thus, this form is $\pi_v$. Therefore, we obtain $(q_v)_\an \simeq \pi_v$ if $q$ is even-dimensional and $(q_v)_\an$ is a subform of $\pi_v$ if $q$ is odd-dimensional.
In both cases we obtain $\dim (p_v-q_v)_\an= A(v)$. It follows that the form $p=f \otimes \pi$ satisfies Properties~\eqref{prop2} and  \eqref{prop3}  for every finite place $v \in \Omega_{2}(F)$.
\end{proof}
By Proposition~\ref{vishik1} $M(\pi)(s)$ is a direct summand of $M(p)$, where $s = (\dim p - \dim \pi)/2$. Let $R$ be the Rost motive of $\pi$. Then $R(s)$ is also a direct summand of $M(p)$.

Recall that ${^v\!R}(t)_{\Fbarv} \simeq N_{\Fbarv}$ for every place $v \in \Omega_{2}(F)$. Since $R$ and ${^v\!R}$ are both Rost motives of some anisotropic $n$-fold Pfister forms, we also have  $ R(t)_{\overline{F}} \simeq N_{\overline{F}}$.
In order to complete the proof (in the case $n>1$), we shall show that $R(t)$ is a direct summand of $M(q)$. By Proposition~\ref{vishik2} this follows from the lemma below.

\begin{lem}
For every field extension $E/F$ the following equivalence holds:
$$i_W(q_E) > t \quad \Longleftrightarrow \quad  i_W(p_E) > s \, , $$
where $i_W(q_E)$ and $i_W(p_E)$ are the corresponding Witt indices over $E$.
\end{lem}
\begin{proof}
 Throughout the proof we use the following observation which follows from formula~\eqref{isotr}. Namely, a twist of a Rost motive $R(i)$ appears in the complete decomposition of $M((q_v)_\an)$ if and only if the motive $R(i+i_W(q_v))$ appears in the complete decomposition of $M(q_v)$.

 \noindent
``$\Longleftarrow$". Assume that $i_W(p_E) > s $ for some field extension $E$ of $F$ (note that $E$ is not necessarily a number field). Recall that $R(s)$ is a binary motive, which is a direct summand of $M(p)$. It follows that $R_E$ is split. Then the motive of the Pfister form $\pi_E$ is also split, since it is a sum of shifts of the motive $R_E$. It follows that $\pi_E$ and, thus, $p_E=f_E \otimes \pi_E$  are split quadratic forms.
Therefore, we have the following equality in the Witt ring $W(E)$\,:
$$q_E=(q-p)_E + p_E = (q-p)_E. \, $$

It follows that $(q_E)_\an \simeq (q_E-p_E)_\an$. Therefore, we have
$$i_W(q_E) = \frac{\dim q_E - \dim (q_E-p_E)_\an}{2} \geq \frac{\dim q - \dim (q-p)_\an}{2} \, .$$
\noindent We claim that the following inequality holds for every $v \in \Omega(F)$\ :
\begin{equation}
\label{ineq1}
\frac{\dim q_v - \dim (q_v-p_v)_\an}{2} > t \, .
\end{equation}
Note that this will imply the necessary inequality  $i_W(q_E)> t$, since by formula~\eqref{hm} we have
${\dim (q-p)_\an = {\max}_{v \in \Omega(F)} \dim (q_v-p_v)_\an}$. In order to prove formula~\eqref{ineq1} we consider two cases: $v \in \Omega_{2}(F)$ and $v \in \Omega_{1+1}(F)$.

Assume $v \in \Omega_{2}(F)$. By the construction of $p$ we have
$$\dim (q_v-p_v)_\an = | \dim (q_v)_\an - \dim(p_v)_\an  |  = | \dim (q_v)_\an - \dim [(q_v)_\an]_{k(v)} |  $$
Therefore, \begin{multline*}
\frac{\dim q_v - \dim (q_v-p_v)_\an}{2} = \frac{ 2i_W(q_v) + \dim (q_v)_\an - | \dim (q_v)_\an - \dim [(q_v)_\an]_{k(v)} | }{2}= \\ = i_W(q_v) + m_0({(q_v)}_\an) + \ldots + m_{k(v)}({(q_v)}_\an) > t  \, .
\end{multline*}
Note that the last equality follows from formula~\eqref{explicit} and the last inequality is a consequence of the decomposition of $M({(q_v)}_\an)$; see formula~\eqref{dec}.

Assume now $v \in\Omega_{1+1}(F)$. Then $^vM$ decomposes into a sum of two Tate motives, and one of them is $\F_2(t)$. Therefore, $\F_2(t)$ is a direct summand of $M(q_v)$. Hence $i_W(q_v)> t$. \\
Since by the construction of $p$, $p_v$ is split, we obtain

$$\frac{\dim q_v - \dim (q_v-p_v)_\an}{2} = \frac{\dim q_v - \dim (q_v)_\an}{2} = i_W(q_v) > t. \, $$ \\

\noindent ``$\Longrightarrow$". Assume that $i_W(q_E) > t $ for some field extension $E$ of $F$. In the Witt ring $W(E)$ we have
$$p_E =(p_E-q_E) + q_E \, .$$
Therefore, \begin{equation}
\label{ineq2}
\dim (p_E)_\an \leq  \dim (p_E -q_E)_\an + \dim (q_E)_\an \leq \dim (p -q)_\an + \dim q - 2t-2
\end{equation}
Note that in order to show $i_W(p_E) > s$, it is sufficient to prove the following inequality for every $v \in \Omega(F)$:
\begin{equation}
\label{ineq3}
\dim (p_v -q_v)_\an  + \dim q_v - 2t-2 < 2^n \, .
\end{equation}
\noindent Indeed, it follows from formula~\eqref{hm} that $\dim (p -q)_\an + \dim q - 2t-2 < 2^n$. Thus, from formula~\eqref{ineq2} we obtain $\dim (p_E)_\an < 2^n$. Hence $i_W(p_E) > (\dim p_E - 2^n)/2=s$.

 In order to prove formula~\eqref{ineq3} we consider two cases: $v \in \Omega_{2}(F)$ and $v \in \Omega_{1+1}(F)$. Assume first $v \in \Omega_2(F)$. Since $^v\!R(t-i_W(q_v))$ is present in the motivic decomposition of the excellent form $(q_v)_\an$ (see \eqref{dec}), we have $t \geq m_0(v)+ \ldots + m_{k(v)-1}(v) + i_W(q_v)$. Therefore,   \begin{multline*}
 2t \geq 2(m_0(v)+ \ldots + m_{k(v)}(v)) - 2 m_{k(v)}(v) + 2i_W((q_v)) \\ = \dim (q_v) - | \dim (q_v)_\an - \dim (p_v)_\an |  - 2 m_{k(v)}(v)\, ,
   \end{multline*}
 \noindent where the equality follows from~\eqref{explicit} and property~\eqref{prop2} of $p$.\\
 Using the above inequality for the left-hand side of \eqref{ineq3} and property~\eqref{prop3} of $p$ we obtain
$$ \dim (p_v -q_v)_\an  + \dim q_v - 2t-2 \leq 2A(v)+2 m_{k(v)}(v) -2 = 2^n -2 <2^n \, ,$$
 \noindent where the last equality follows from the explicit expression of $A(v)$ and $m_{k(v)}(v)$; see Section~\ref{excellent-dec}.

 Assume now that $v \in \Omega_{1+1}(F)$. Then, by construction, the form $p_v$ is split and we have
 ${\dim (p_v -q_v)_\an = \dim (q_v)_\an = \dim q_v - 2i_W(q_v)}$.

 Since $^v\!M$ is split, we have  $^v\!M \simeq \F_2(t) \oplus \F_2(t + 2^{n-1} -1)$. Therefore, the Tate motive $\F_2(t + 2^{n-1} -1)$ is present in the motivic decomposition of $p_v$, and, by duality (see \cite[\S 65]{EKM}), the same is true for the trivial Tate motive $\F_2$ twisted by
 $$\dim q_v -2 - (t + 2^{n-1} -1)= \dim q_v - t - 2^{n-1} - 1 \, . $$
 \noindent It follows that $i_W(q_v) > \dim q_v - t - 2^{n-1} - 1$. Using this we obtain the necessary inequality~\eqref{ineq3}. Indeed,
 \begin{multline*}
 \dim (p_v -q_v)_\an  + \dim q_v - 2t-2 = 2\dim q_v - 2i_W(q_v) -2t -2 \\ < 2\dim q_v - 2(\dim q_v - t - 2^{n-1} - 1) -2t -2 = 2^n.
 \end{multline*}
\end{proof}
In order to complete the proof of the theorem it remains to consider the case $n=1$.
Note that in this case the quadratic form $q$ is even-dimensional: $\dim q = 2d+2$.
For every $v\in \Omega_{1+1}(F)$ the motive $^v\!M$ is split, ${^v\!M \simeq \F_2(d) \oplus \F_2(d)}$, and therefore the form $q_v$ is split as well.
Let $v\in \Omega_{2}(F)$. Then the motive $^v\!M$ is indecomposable. It follows that the quadratic form $q_v$ has a non-trivial discriminant, and
${^v\!M \simeq M(\Spec F_v(\sqrt{\disc q_v}\hs)\hs)(d)}$; see Section~\ref{excellent-dec}.
In both cases the following equivalence holds: for every $v \in \Omega(F)$ and for every field extension $E$ of $F_v$  we have
\begin{equation}
\label{disc}
(q_v)_E \, \text{ is split} \quad \Longleftrightarrow \quad (q_v)_E \, \text{ has trivial discriminant}. \,
\end{equation}
Note that $ \Omega_{2}(F)$ is non-empty, since $q$ is anisotropic. Therefore, the discriminant of $q$ is non-trivial and it becomes trivial over the quadratic field extension $K=F(\sqrt{\disc q})$ of $F$.

We claim that the quadratic form $q$ is split over $K$. Since $K$ is a number field, by the Hasse--Minkowski theorem (see formula~\eqref{hm}), it is sufficient to show that $(q_{K})_w$ is split for every $w \in \Omega(K)$. Note that $K_w$ is an extension of $F_v$ for some $v \in \Omega(F)$. Since $\disc q_{K_w}$ is trivial, the claim follows from~\eqref{disc}.

It follows that $q$ becomes hyperbolic over the function field of the $1$-fold Pfister form $\pi = \langle 1, - \disc q \rangle$. Then, by \cite[Theorem~VII.3.2]{Lam} the form $q$ is divisible by $\pi$. Finally, by Proposition~\ref{vishik1} $M(\pi)(d)$ is a direct motivic summand of $q$.
\end{proof}

As an immediate corollary of the proof we obtain:
\begin{cor}\label{c1}
Let $F$ be a number field and let $q$ be a quadratic form over $F$.
Then every indecomposable binary direct summand of $M(q)$ is a twist of a Rost motive.
\end{cor}

\begin{rem}
Corollary~\ref{c1} also follows from \cite[Theorem~6.9]{IzhV00}, where motivic cohomology of simplicial varieties and the Milnor operations are used.
Conjecturally Corollary~\ref{c1} holds over any field. Voevodsky formulated in \cite[Remark~5.4]{Vo11} an even more general conjecture.
\end{rem}

Recall, that the motive of a quadric decomposes uniquely into a sum of indecomposable direct summands. Applying the main theorem we can describe the complete motivic decomposition of a quadric over every number field which has at most one real embedding. In the corollary below we consider only anisotropic quadratic forms, since by formula~\eqref{isotr} the description of the complete motivic decomposition of a quadric can be always reduced to the anisotropic case.

For a geometrically split motive $S$ we say that it has rank $r$, if $S$ is a sum of $r$ Tate motives over its splitting field.

\begin{cor}\label{c2}
Let $q$ be an anisotropic quadratic form over a number field $F$. Assume that $F$ has at most one real embedding. Then the motive $M(q)$ decomposes into a direct sum of Rost motives and a motive $S$, which is either $0$, or is indecomposable of rank $4$, $6$ or $8$, or $S \simeq S' \oplus S'(1) $, where $S'$ is indecomposable of rank $4$.

Moreover, one can explicitly recover
the decomposition of $S$ (and of $S'$) into a direct sum of Tate motives over a splitting field of $q$ from the set of Witt indices $\{i_W(q_v) \mid v \in \Omega(F)\}$.
\end{cor}
\begin{proof}
Without loss of generality we may assume that
$\dim q\ge 4$ and $F$ has one real place.

First we apply Theorem~\ref{maintheorem} to all split binary motives $N$ and split off all binary direct summands in $M(q)$.
Let us denote the remaining direct summand as $S$ and assume that it is non-zero.

Let first $\dim q=2d+1\ge 5$. Over the finite places the motive of the form $q$ is either split or contains an indecomposable direct summand, which over a splitting field of $q$ is isomorphic to $\F_2(d-1)\oplus\F_2(d)$. Over the real place the motive of $q$ contains indecomposable direct summands, which over a splitting field of $q$ become isomorphic to $\F_2(s)\oplus\F_2(d)$ and $\F_2(d-1)\oplus\F_2(2d-s-1)$ respectively with some $0\le s\le d-2$, $d=s+2^r-1$ for some $r\geq 2$. Therefore $S$ over $\overline F$ is isomorphic to $$\F_2(s)\oplus\F_2(d-1)\oplus\F_2(d)\oplus\F_2(2d-s-1),$$
and $S$ is indecomposable over $F$.

Similarly, if $\dim q=2d+2$, then over the finite places the motive of $q$ is either split or contains an indecomposable direct summand, which over a splitting field of $q$ is isomorphic to $\F_2(d)\oplus\F_2(d)$, or contains indecomposable summands, which over a splitting field of $q$ become isomorphic to $\F_2(d-1)\oplus\F_2(d)$ and $\F_2(d)\oplus\F_2(d+1)$ respectively. Over the real place the motive of $q$ contains either indecomposable direct summands, which over a splitting field of $q$ are isomorphic
to $\F_2(s)\oplus\F_2(d+1)$, $\F_2(d-1)\oplus\F_2(2d-s)$, and
$\F_2(d)\oplus\F_2(d)$ (for some $0\le s\le d-2$, $d=s+2^r-2$ for some $r \geq 2$), or which are isomorphic over a splitting field of $q$ to $\F_2(s)\oplus\F_2(d)$ and $\F_2(d)\oplus\F_2(2d-s)$ (for some $0\le s\le d-1$, $d=s+2^r-1$ for some $r \geq 1$), or which are isomorphic over a splitting field of $q$ to $\F_2(s)\oplus\F_2(d)$, $\F_2(d)\oplus\F_2(2d-s)$, $\F_2(s+1)\oplus\F_2(d+1)$, and $\F_2(d-1)\oplus\F_2(2d-s-1)$ (for some $0\le s\le d-2$, $d=s+2^r-1$ for some $r \geq 2$).

Thus, if $\dim q=2d+2$, then $S$ over $\overline F$ is isomorphic either to
$$\F_2(s)\oplus\F_2(d)^{\oplus 2}\oplus\F_2(2d-s)$$ (for some $0\le s\le d-1$, $d=s+2^r-1$ for some $r\geq 1$)
or to
$$\F_2(s)\oplus\F_2(d-1)\oplus\F_2(d)^{\oplus 2}\oplus\F_2(d+1)\oplus\F_2(2d-s)$$ (for some $0\le s\le d-2$, $d=s+2^r-2$ for some $r \geq 2$)
or to
$$\widetilde N:=\F_2(s)\oplus\F_2(s+1)\oplus\F_2(d-1)\oplus\F_2(d)^{\oplus 2}\oplus\F_2(d+1)\oplus\F_2(2d-s-1)\oplus\F_2(2d-s)$$
(for some $0\le s\le d-2$, $d=s+2^r-1$ for some $r \geq 2$), cf. Remark~\ref{pic} below.

In the first two cases the motive $S$ is indecomposable over $F$ by the same reasons as in the odd-dimensional case.
In the third case the motive $S$ is indecomposable over $F$, if the discriminant of $q$ is non-trivial, since in this case the motive of $q$ over some finite place contains an indecomposable direct summand that over a splitting field becomes isomorphic to $\F_2(d)^{\oplus 2}$.

Therefore, it remains to consider the case when the discriminant of $q$ is trivial and $S$ over $\overline F$ is isomorphic to $\widetilde N$.
Note that there is a (unique) motive $N'$ such that $\widetilde N=N'\oplus N'(1)$, and by the above considerations the motive $S$ is a direct sum of at most two indecomposable summands.
We claim that the motive $S$ is actually decomposable
as $S'\oplus S'(1)$, where $S'$ is indecomposable over $F$ and $S'$ is isomorphic to $N'$ over $\overline F$.

From now on we assume without loss of generality that ${q}_{\R}$ is positive definite. Since  $S$ over ${\overline F}$ is isomorphic to $\widetilde{N}$, we have $\dim q = 2^n m$, where $m$ is odd,  $n=r+1 \geq 3$ and $s= d -2^r+1$.

First we shall reduce the problem to the case $\dim q = 2^n$ (i.e. to the case $s=0$). There exists a quadratic form $\tilde{q}$ of dimension $2^n$ over $F$ such that $\tilde{q}_{\R} \simeq \langle 1,1 \rangle ^{\otimes n}$ and $(\tilde{q}_v)_{\an} \simeq (q_v)_{\an}$ for every finite place $v \in \Omega(F)$. Indeed, the existence of $\tilde{q}$ follows from Proposition~\ref{existence}, since the equalities $\varepsilon_v(\tilde{q}) = \varepsilon_v(q)$ and $d(\tilde{q}_v)=1$ hold for every place $v \in \Omega(F)$. Note that $s = d -2^{n-1}+1 =  (\dim q -\dim \tilde{q})/2$. Let $U$ be the upper indecomposable motive of $M(\tilde{q})$ (see \cite[\S2b]{upper}). Lemma~\ref{lupper} below shows that $U(s)$ is a direct summand of $M(q)$. It follows that we can replace $q$ by $\tilde{q}$, that is we can assume $\dim q = 2^n$.

Before we prove the lemma, let us note that the following isomorphisms hold for the forms $q$ and $\tilde{q}$:
$$\langle 1, 1 \rangle \otimes \tilde{q} \simeq \pi \quad \text{  and  } \quad q \simeq \tilde{q} \otimes f  \simeq  \tilde{q} \perp (\pi \otimes g) \, , $$
\noindent where $\pi = \langle 1, 1 \rangle^{\otimes (n+1)}$, $f =  \langle 1 \rangle^{\oplus m} $ and $g =  \langle 1 \rangle^{\oplus (m-1)/2}$. One can check each of these isomorphisms locally for every $v \in \Omega(F)$. For a unique real place the isomorphisms clearly hold. In order to check that they also hold for an arbitrary finite place $v \in \Omega(F)$, note that the $(n+1)$-fold Pfister form $\pi_v$ and the form $\langle 1, 1 \rangle \otimes \tilde{q}_v$ are hyperbolic (recall that the form $\tilde{q}_v$ is either hyperbolic or $(\tilde{q}_v)_{\an}$ is an anisotropic $2$-fold Pfister form).

\begin{lem}\label{lupper}
The motive $U(s)$ is a direct summand of $M(q)$.
\end{lem}
\begin{proof}
By Proposition~\ref{vishik2} it is sufficient to show that for every field extension $E/F$ the following equivalence holds:
$$i_W(q_E) > s \quad \Longleftrightarrow \quad  i_W(\tilde{q}_E) > 0 \, . $$

\noindent ``$\Longleftarrow$". Assume that $i_W(\tilde{q}_E) > 0$ for some field extension $E$ of $F$. Then the $(n+1)$-fold Pfister form ${\pi_E \simeq \langle 1, 1 \rangle\otimes\tilde{q}_E}$ is isotropic and, therefore, is hyperbolic over $E$. Since $q \simeq \tilde{q} \perp (\pi \otimes g) $, we get $q_E =\tilde{q}_E$ in the Witt ring $W(E)$. It follows that $i_W(q_E) > s$.

\noindent ``$\Longrightarrow$". Assume that $i_W(q_E) > s$ for some field extension $E$ of $F$. Then the form $q_E$ can be represented in $W(E)$ by  a form of dimension $< 2^n$. Therefore the form $\langle 1, 1 \rangle \otimes q_E$  can be represented in $W(E)$ by a form of dimension $< 2^{n+1}$.
Note that $$\langle 1, 1 \rangle \otimes q_E \simeq \pi_E \otimes f_E \in I^{n+1}(E),$$ where $I$ denotes the fundamental ideal in the Witt ring. It follows from \cite[Hauptsatz~X.5.1]{Lam} that $\pi_E \otimes f_E = 0$ in $W(E)$, that is the form  $\pi_E \otimes f_E$ is hyperbolic. By \cite[Cor.~3.3.4]{Kahn} the $(n+1)$-fold Pfister form $\pi_E$ is also hyperbolic over $E$. Finally, since $q \simeq \tilde{q} \perp (\pi \otimes g)$, we get $q_E =\tilde{q}_E$ in $W(E)$. Now it follows from $i_W(q_E) > s$ that $i_W(\tilde{q}_E) > 0$.
\end{proof}

Next we prove the following lemma.

\begin{lem}\label{lwitt}
The first Witt index of $q$ is greater than $1$.
\end{lem}
\begin{proof}
Consider the quadratic form  $\varphi \perp (-q)$, where  $\varphi = \langle 1, 1 \rangle^{\otimes n}$, and denote by $\rho$ the anisotropic part of this form. It is easy to see that $\rho$ has dimension $4$, trivial discriminant and represents $1$. In particular, $\rho$ is a $2$-fold Pfister form and we can write $\rho = \langle\! \langle a,b \rangle\! \rangle$ for some elements $a, b \in  F^*$ with $a_{\R}<0$ (indeed, if $a_{\R}>0$ and $b_{\R}>0$, note that $\langle\! \langle a,b \rangle\! \rangle \simeq \langle\! \langle -ab,a \rangle\! \rangle$).

Since the both forms $\rho$ and $\varphi = \langle 1, 1 \rangle^{\otimes n} $ become hyperbolic over $F(\sqrt{a})$, the form $q$ is also hyperbolic over $F(\sqrt{a})$.
Hence, by \cite[Theorem~VII.3.2]{Lam} the form $q$ is divisible by $\langle 1, -a \rangle$.

To finish the proof observe that every isotropic quadratic form $g$ of the type $g=\langle 1, -a \rangle \otimes f$ for some form $f$ of dimension bigger than $1$ has Witt index at least $2$. Indeed, the form $\langle 1, -a \rangle \otimes f$ is hyperbolic, if $a$ is a square.
Otherwise, we can write $g(x,y)=f(x)-af(y)=0$ for some non-zero vector $(x,y)$, and $g(ay,x)=0$ with $(ay,x)$ not collinear to $(x,y)$; see also \cite[Lemma~5.2]{vish-lens} for a more general statement.
\end{proof}

It follows now from Lemma~\ref{lwitt} and \cite[Theorem~4.13]{vish-lens} that the motive $S$ from the beginning of the proof of the
present corollary is isomorphic over $F$ to $S'\oplus S'(1)$ with $S'=U(s)$.
\end{proof}

\begin{rem}\label{pic}
In \cite[Section~4]{vish-lens} Vishik visualizes motivic decompositions of quadrics using diagrams, where the dots represent Tate motives and the arcs represent connections between them in the motive of a quadric. Using Vishik's diagrams we can represent all possibilities for the motive $S$ from Corollary~\ref{c2} graphically as follows:

$$
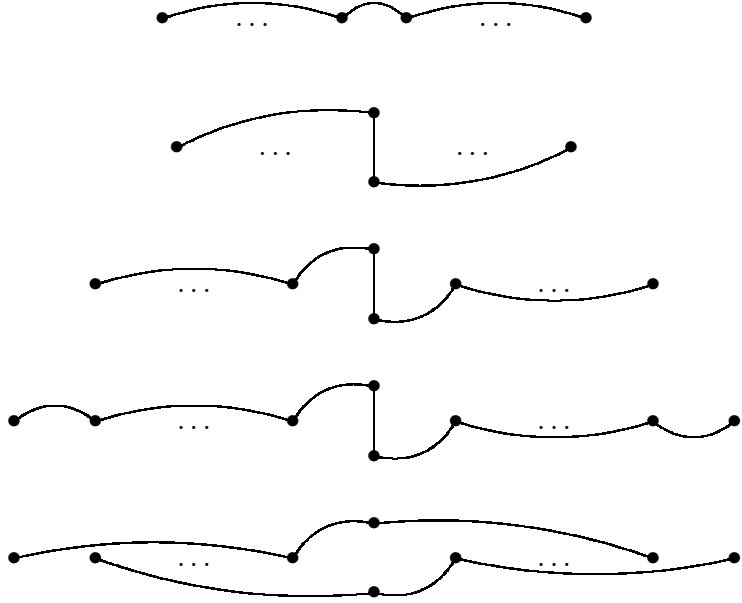
\xymatrix{
*=0{\bullet} \ar@{-}@/^/[rr]!U & \ldots     &     *=0{\bullet} \ar@{-}@/^/[r]!U & *=0{\bullet}\ar@{-}@/^/[rr]!U & \ldots &*=0{\bullet}
 }
$$

$$\xymatrix@R-1.5pc{
 && *=0{\bullet}\ar@{-}[dd] && \\
*=0{\bullet} \ar@{-}@/^/[rru]!U & \ldots       &  & \ldots & *=0{\bullet} \\
&& *=0{\bullet} \ar@{-}@/_/[rru]!U &&
 }
$$

$$\xymatrix@R-1.5pc{
 &&& *=0{\bullet}\ar@{-}[dd] &&& \\
*=0{\bullet} \ar@{-}@/^/[rr]!U & \ldots     & *=0{\bullet}\ar@{-}@/^/[ru]!U &  & *=0{\bullet}\ar@{-}@/_/[rr]!U & \ldots & *=0{\bullet} \\
&&& *=0{\bullet} \ar@{-}@/_/[ru]!U &&&
 }
$$

$$\xymatrix@R-1.5pc{
 &&&& *=0{\bullet}\ar@{-}[dd] &&&& \\
*=0{\bullet} \ar@{-}@/^/[r]!U & *=0{\bullet}\ar@{-}@/^/[rr]!U & \ldots     & *=0{\bullet}\ar@{-}@/^/[ru]!U &  & *=0{\bullet}\ar@{-}@/_/[rr]!U & \ldots & *=0{\bullet}\ar@{-}@/_/[r]!U & *=0{\bullet}\\
&&&& *=0{\bullet} \ar@{-}@/_/[ru]!U &&&&
 }
$$

$$\xymatrix@R-1.5pc{
 &&&& *=0{\bullet}\ar@{-}@/^/[rrrd]!U &&&& \\
*=0{\bullet} \ar@{-}@/^/[rrr]!U & *=0{\bullet}\ar@{-}@/_/[rrrd]!U & \ldots     & *=0{\bullet}\ar@{-}@/^/[ru]!U &  & *=0{\bullet}\ar@{-}@/_/[rrr]!U & \ldots & *=0{\bullet} & *=0{\bullet}\\
&&&& *=0{\bullet} \ar@{-}@/_/[ru]!U &&&&
 }
$$

The last diagram represents the splitting of $S$ as a direct sum $S'\oplus S'(1)$.
\end{rem}

\begin{exple}
Let $q$ be an $11$-dimensional quadratic form over a number field $F$. Assume that the only values of the Witt index of $q$ over the real places
are $0$ and $2$, that is
$$\{i_W( q_v) \mid v \in \Omega_{\R}(F)\} = \{0,2\} \, . $$
Then we can find a complete decomposition of $M(q)$. Indeed, let $v$ be a real place of $F$. If $i_W(q_v)=0$, we have
\begin{equation}
\label{11-dim0}
M(q_v)= R_4 \oplus R_4(1) \oplus R_4(2) \oplus R_3(3)\oplus R_2(4)
\end{equation}
\noindent and if  $i_W(q_v)=2$, then
\begin{equation}
\label{11-dim2}
M(q_v)= \F_2 \oplus \F_2(1) \oplus R_3(2) \oplus R_3(3) \oplus R_3(4) \oplus \F_2(8) \oplus \F_2(9) \, ,
\end{equation}
\noindent where $R_i$ is the Rost motive of a unique anisotropic $i$-fold Pfister form over $F_v =\R$,
$$(R_i)_{\Fbarv} \simeq \F_2 \oplus \F_2(2^{i-1}-1).$$

If $v$ is a finite place of $F$, then $q_v$ is either split or $\dim (q_v)_{\an} =3$. In the first case the motive $M(q_v)$ is split. In the second case $(q_v)_{\an}$  is a subform of a unique anisotropic 2-fold Pfister form over $F_v$. Denote by $R_v$ the Rost motive of this Pfister form, $(R_v)_{\Fbarv} \simeq \F_2 \oplus \F_2(1) $. Then $M((q_v)_{an}) \simeq R_v$ and we have
\begin{equation}
\label{11-dim3}
M(q_v)=  {\oplus}_{i=0}^{3}( \F_2(i)\oplus \F_2(9-i)) \oplus R_v(4)  \, .
\end{equation}

Note that the hypotheses of Theorem~\ref{maintheorem} hold for $N= \F_2(1) \oplus \F_2(8) $ and for $N= \F_2(3) \oplus \F_2(6)$.
It follows that $M(q)$ has two direct summands $\widetilde{R_4}(1)$ and $\widetilde{R_3}(3)$, where $\widetilde{R_i}$ is the Rost motive of some
anisotropic $i$-fold Pfister form over $F$. Therefore, for some motive $U$ we have
\begin{equation}
\label{complete}
M(q) = U \oplus \widetilde{R_4}(1) \oplus \widetilde{R_3}(3) \, .
\end{equation}
It follows from decompositions~\eqref{11-dim0} and \eqref{11-dim2} that the motive $U$ is indecomposable. Thus, decomposition~\eqref{complete}
of the motive $M(q)$ is complete.
Note that the motive $U$ in the decomposition is the upper motive of $M(q)$ in terms of \cite[\S2b]{upper}.
\end{exple}

\bigskip
{\sc Acknowledgments.}
The authors are grateful to David Leep for his help with quadratic forms over number fields.

\end{document}